\documentclass{interact}
\usepackage{graphicx}
\usepackage{color}
\usepackage{latexsym}
\usepackage{subfigure}
\usepackage{amscd}
\usepackage{a4wide}
\usepackage{epsfig}
\usepackage{float}
\usepackage{amsmath, amssymb, amsthm, amsfonts, amsgen} 
\usepackage{epstopdf}
\usepackage{mathtools}
\usepackage{enumitem}
\usepackage{etoolbox}
\usepackage{bm}
\usepackage{comment}

\graphicspath{ {./} }

\numberwithin{equation}{section}
\newtheorem{definition}{Definition}[section]
\newtheorem{lemma}[definition]{Lemma}
\newtheorem{theorem}[definition]{Theorem}

\newtheorem{corollary}[definition]{Corollary}
\newtheorem{remark}[definition]{Remark}

\renewcommand{\sb}[2]{#1(#2)}
\newcommand{\sinf}[1]{\left(\sb{#1}{n}\right)_{n\in\mathbb{N}}}

\newcommand{\E}{\mathcal{E}}
\newcommand{\F}{\mathcal{F}}
\newcommand{\G}{\mathcal{G}}

\newcommand{\underrel}[2]{\mathrel{\mathop{#2}\limits_{#1}}}
\newcommand{\Dz}{$\mathrm{D}_0$}
\newcommand{\D}{$\mathrm{D}$}
\newcommand{\T}{\mathcal{T}}
\newcommand{\U}{\mathcal{U}}

\newcommand{\flg}[1]{\lfloor\gamma #1\rfloor}
\newcommand{\flgg}[1]{\lfloor\gamma^2 #1\rfloor}
\newcommand{\fr}[1]{\{#1\}}

\begin{document}
\parindent 0pt

\title{A diluted version of the problem of the existence of the Hofstadter sequence}

\author{
\name{Jonathan H.B.\ Deane\textsuperscript{a}\thanks{Contact: J.H.B.\ Deane, e-mail: j.deane@surrey.ac.uk}
and Guido Gentile\textsuperscript{b}}
\affil{\textsuperscript{a} Department of Mathematics, University of Surrey, Guildford, GU2 7XH, UK;
\textsuperscript{b} Dipartimento di Matematica, Universit\`a Roma Tre, Roma, I-00146, Italy} }
\date{}

\maketitle

\begin{abstract}
We investigate the conditions on an integer sequence
$f(n)$, $n\in\mathbb N$, with $f(1) = 0$, such that the sequence $q(n)$,
computed recursively via $q(n) = q(n-q(n-1)) + f(n)$, with $q(1) = 1$, exists.
We prove that $f(n+1) - f(n)\in\{0, 1\}$, $n\geq 1$, is a sufficient but not
necessary condition for the existence of sequence $q$.

Sequences $q$ defined in this way typically display non-trivial
dynamics: in particular, they are generally aperiodic with no
obvious patterns. We discuss and illustrate this behaviour with some
examples.
\end{abstract}

\begin{amscode}
11B37, 11B39
\end{amscode}

\section{Motivation}
In his 1979 book~\cite{geb}, author D.R.\ Hofstadter mentions an integer sequence, $q_h(n)$, defined,
for $n>2$, by
\begin{equation}
\label{hof}
q_h(n) = q_h(n-q_h(n-1)) + q_h(n-q_h(n-2))\qquad\text{ with }\qquad q_h(1) = q_h(2) = 1.
\end{equation}
In order for $q_h(n)$ to exist for all $n\in\mathbb N$, it must be true,
again for all $n\in\mathbb N$, that $1\leq q_h(n) \leq n$, since only then does the right-hand side
of~\eqref{hof} refer to terms that are already known: $q_h(n)$ is undefined for
$n\leq 0$. Should it happen that for some $n$, $q_h(n)>n$, the sequence is
said to die at $n$. Although direct computation shows that $q_h(n)$ exists for $1\leq
n\leq 3\times 10^{10}$~\cite{oeis_a}, the existence of $q_h(n)$ for all
positive $n$ is still an open question.

The sequence $q_h(n)$ is non-monotonic, aperiodic and its dynamical behaviour
appears to be complex.
In~\cite{pinn} for instance, small scale `chaotic' behaviour is described, with some order apparent at larger 
scales, and several statistical properties of the sequence are also investigated.

This complex behaviour has inspired several authors to study variations of Hofstadter's original
recursion. For instance, Tanny~\cite{tanny92} considers
$$T(n) = T(n-1-T(n-1)) + T(n-2-T(n-2)) \mbox{ with } T(0) = T(1) = T(2) = 1,$$
for $n>2$~\cite{oeis_b}. Another variant is considered in~\cite{balamohan}, in which the
authors investigate
$$V(n) = V(n - V(n - 1)) + V(n - V(n - 4)) \mbox{ with } V(1) = V(2) = V(3) = V(4) = 1$$
for $n>4$~\cite{oeis_c}.  Both these variants give rise to sequences that are monotonic and hit every 
positive integer, neither of which is true for $q_h(n)$.

A different approach is followed in~\cite{fox}, in which the recursion formula
in~\eqref{hof} is retained but different initial conditions are used. For certain initial conditions,
`eventually quasi-polynomial' solutions can be found. For a given
integer $m>0$ and a given fixed set of polynomials $p_i(n)$, $i = 0,\ldots, m-1$,
a quasi-polynomial function of $n$, $h(n)$, say, is defined by 
$h(n) = p_i(n)$, where $i = n\bmod m$. If this holds only for $n$ greater
than some positive integer $n_0$ say,
then $h(n)$ is said to be eventually quasi-polynomial. In~\cite{fox}, families of 
eventually quasi-polynomial solutions to 
\begin{equation}
\label{qpex}
r(n) = r(n-r(n-1)) + r(n - r(n-2)),
\end{equation}
with suitable initial conditions, are constructed.  As an
example of such a solution with $m=5$ and where the $p_i(n)$ are of degree at most 1, we give

\begin{center}
\parbox{0pt}{\begin{tabbing}
$p(i)\qquad\qquad$  \= $n-4\qquad$  \= $n-4\qquad$  \= $n-4\qquad$   \= $n-4\qquad$  \= $n-4\qquad$\kill
$i$     \> 0    \> 1      \> 2     \> 3      \> 4\\
$p_i(n)$  \> $2$  \> $n-4$  \> $5$   \> $n-5$  \> $n-6$,
\end{tabbing}}
\end{center}

along with the initial conditions $r(3),\ldots, r(12) = 1, 1, 3, 5, 1, 4, 7, 6, 4, 9$.
(Values of $r(1)$ and $r(2)$ are not needed to compute $r(n), n > 2$.)
For all $n> n_0 = 12$ here, $r(n)$ satisfying~\eqref{qpex} is then given by the quasi-polynomial above.

In this paper, we also consider a variation on the original problem, one
which consists of an infinite number of variants rather than just a single case. In
particular, we discuss a `diluted' version of the problem of the
existence of $q_h(n)$. In order to describe this version, we first
need the notation $\sinf{a} := a(1), a(2), \ldots$ for an
infinite sequence. We use this notation in formal contexts, such as
in definitions and lemmas, but where it it is clear what is
intended, we just write $a$ to mean the whole sequence. Naturally, when the
$n$-th term is meant, we write $a(n)$.  We also need the definitions
\begin{definition}[Property D]
If integer sequence $\sinf{a}$ obeys $a(n+1)-a(n)\in\{0, 1\}$ for $n\in\mathbb N$, 
then the sequence is said to have property \D.
\end{definition}
\begin{definition}[Property \Dz]
If integer sequence $\sinf{a}$ has property \D\ and additionally $a(1) = 0$,
then it is said to have property \Dz.
\end{definition}
Sequences that have property D or \Dz\ are often described as `slow' sequences.

The problem that we consider is the question of the existence of sequences
$q$ defined recursively by
\begin{equation}
\label{rec}
\sb{q}{n} = \sb{q}{n - \sb{q}{n-1}} + \sb{f}{n}\qquad\mbox{with}\qquad q(1) = 1,
\end{equation}
for $n\in\mathbb N$, $n>1$, and where $\sb{f}{n}$ has property \Dz\ (so in
particular $f(1) = 0$).
All sequences in the paper start from index 1 and we apply the conditions $q(1)=1$ and 
$f(1) = 0$ strictly throughout, even though $f(1)$ is never needed to compute the sequence $q$.
Computations suggest that, under these conditions, the sequence $q$ corresponding
to any $f$ with property \Dz\ exists for all $n\in\mathbb N$.

We think of sequences generated by~\eqref{rec} as `diluted' Hofstadter
sequences because of the similarity of their definition to that of the original $q_h$-sequence.
In particular, in~\eqref{rec}, one `difficult' --- nested --- term is retained, $q(n-q(n-1))$, but the second
nested term is replaced by $f(n)$, which is under our control.

Clearly the existence of the sequence $q$ is also
equivalent to the inequality $1 \leq q(n) \leq n$ holding for all $n\in\mathbb N$.
If this is the case, then we simply say that $q$ exists.
The lower bound on $q(n)$ here is easy: by~\eqref{rec}, for any $n$ for
which $q$ exists, $q(n)$ is the sum of a positive integer and a non-negative
one, and so is also greater than or equal to unity.

The upper bound on $q$, however, is less obvious.
In what follows, after some preliminaries, we give a proof of the following theorem:
\begin{theorem}
\label{main}
For all sequences $\sinf{f}$ having property \Dz, the corresponding sequence
$\sinf{q}$ with $\sb{q}{1} = 1$ and with $\sb{q}{n}$, for $n>1$, computed
from~\eqref{rec}, exists for all $n\in\mathbb N$.
\end{theorem}
The proof of this theorem is the main result of the paper, but the
recursion~\eqref{rec} can also be viewed as a non-linear (in the light of the
nested term), non-autonomous (because of $f(n)$) dynamical system, discrete in both time and space.
Hence, having proved the existence of $q$ for all $f$ with property \Dz, we make some observations 
about the dynamics of sequences $q$, which are also interesting.

\section{Preliminaries}

The first term of all sequences in this paper has index 1.
We start by giving the following definitions:

\begin{itemize}
\item Two sequences $\sinf{a}$ and $\sinf{b}$ \textit{differ} if there
exists $n'\in\mathbb N$ such that $\sb{a}{n'}\neq \sb{b}{n'}$;
otherwise, they are equal. 
\item A sequence $\sinf{q}$ obeying~\eqref{rec} \textit{exists} if and only if $\sb{q}{n}$ is 
defined for all $n\in\mathbb N$; or, equivalently, if $1\leq q(n)\leq n$ for all $n\in\mathbb N$.
\item A sequence $Q(f) := \sinf{q}$ is \textit{the sequence corresponding to sequence}
$\sinf{f}$ if $\sinf{q}$ exists and is defined by~\eqref{rec}.
\end{itemize}
Where an explicit expression for $\sb{f}{n}$ is available, for example $\sb{f}{n} = \lfloor n/2\rfloor$,
we may write $Q(\lfloor n/2\rfloor)$.

Where we need to refer explicitly to the first few terms of a sequence, $a(n)$ say, we list the values
starting from $a(1)$ --- for example $a = (1, 2, 4, 4)$ means $a(1) = 1$, $a(2) = 2$, $a(3) = a(4) = 4$.

Let us define $\F$ as the set of all sequences having property \Dz: that is
$$\F:=\left\{\sinf{f}: f(1) = 0, \sb{f}{i+1}-\sb{f}{i}\in\{0,1\}\mbox{ for }i\in\mathbb N\right\}.$$
We will also need the set of finite sequences with property \Dz:
$$\F_m:=\{\left(f(n)\right)_{n=1,\ldots,m}: f(1) = 0, \sb{f}{i+1}-\sb{f}{i}\in\{0,1\}\mbox{ for }i=1,\ldots,m-1\}.$$

It is easily established that the cardinality of $\F$ is the same
as that of the real numbers. To do so, let $f\in\F$ and consider the sequence of differences
$\left(\sb{f}{i+1}-\sb{f}{i}\right)_{i\in\mathbb N}$, which will be a sequence consisting
of the symbols 0 and 1. Interpreting this as a binary decimal, we see that each distinct
$f\in\F$ gives rise to a distinct real number $x\in[0,1]$.

One important subset of $\F$, with same cardinality, is
$$\G := \left\{\sinf{f} = \lfloor\alpha n\rfloor \mbox{ with } \alpha\in[0,1)\right\}.$$
It is easy to see (i) that $\G\subset\F$, and that in
fact (ii) $\G\subsetneq\F$. Property (i) comes directly from the
definition $\lfloor x\rfloor := \max\left\{k\in\mathbb Z : k\leq x\right\}$,
from which we deduce that, for $\alpha\in[0,1]$,
$\lfloor\alpha(n+1)\rfloor - \lfloor\alpha n\rfloor$ is either 0 or 1.
For property (ii), consider the sequence $\sb{a}{1},\ldots, \sb{a}{4} = 0, 0, 1, 2$,
which are the first four terms of a sequence in $\F$.
Now assume that these terms are generated by $\sb{a}{n} = \lfloor \alpha n\rfloor$ for some $\alpha$.
Then $\sb{a}{2} = 0\implies \alpha\in[0,1/2)$, but $\sb{a}{4} = 2\implies \alpha\in[1/2,3/4)$,
leading to a contradiction. Hence, there are sequences in
$\F$ that are not in $\G$.

Furthermore, the condition that successive terms in a sequence $f(n)$ differ by at
most unity turns out to be sufficient but not necessary for the corresponding
$\sinf{q}$ to exist. There are examples where $f$ increases by more than
unity, and where it is not monotonic, but for which $Q(f)$ nonetheless exists. Consider, for instance, 
(a) $f_a = (0,0,2,2,4,4,\ldots)$, for which $q_a = Q(f_a) = (1, 1, 3, 3, 5, 5,\ldots)$ and
(b) $f_b = (0,1,0,1,0,1,\ldots)$, for which $q_b = Q(f_b) = (1, 2, 1, 2, 1, 2,\ldots)$.
Statements (a) and (b) are proved in
\begin{lemma}
\label{nonFq}
If
$$f_a(n) = 2\left\lfloor{\frac{n-1}{2}}\right\rfloor = n - \frac{3 + (-1)^n}{2},
\;\;\mbox{ then  }\;\; q_a(n) = n - \frac{1+(-1)^n}{2}$$
and if
$$f_b(n) = \frac{1+(-1)^n}{2},\;\;\mbox{ then  }\;\; q_b(n) = \frac{3+(-1)^n}{2}$$
for all $n\in\mathbb N.$ 
\end{lemma}
\begin{proof}
These just boil down to computation.

For $q_a$, we have that $n-q_a(n-1) = 1$ ($n$ even), $=2$ ($n$ odd), so $q_a(n-q_a(n-1)) = 1$
for all $n$. Hence, $q_a(n) - q_a(n-q_a(n-1)) = q_a(n)-1 = n - \frac{3 + (-1)^n}{2} = f_a(n)$,
as claimed.

For $f_b$, we have $n-q_b(n-1) = n-1$ ($n$ even), $n-2$ ($n$ odd). Hence,
$n-q_b(n-1)$ is odd, and so $q_b(n-q_b(n-1)) = 1$, both for all $n$.
Therefore, $q_b(n) - q_b(n-q_b(n-1)) = q_b(n)-1 =\frac{1+(-1)^n}{2} = f_b(n)$.
\end{proof}
\begin{remark}
Example (b) is in fact a special case of 
$$\mbox{If $m\in\mathbb N$ and }  f(n) = (n-1)\bmod m,\mbox{ then } q(n) = f(n)+1
= ((n-1)\bmod m) + 1.$$
\end{remark}
On the other hand, it is easy to find examples of $f\notin\F$ for which $Q(f)$
does not exist --- for instance, $f= (0, 2, 2)$,
for which $q(1) = 1, q(2) = 3$ and so $q(3) = q(0)+f(3)$ which is undefined.

We also prove
\begin{lemma}
\label{differ}
If $f$ and $f'$ are different sequences,
and $Q(f)$ and $Q(f')$ both exist, then $Q(f)\neq Q(f')$.
\end{lemma}
\begin{proof}
Let $f(n) = f'(n)$ for $1\leq n\leq k$ but $f(k+1)\neq f'(k+1)$.
Write $q = Q(f)$, $q' = Q(f')$. Now, $q(n) = q'(n)$ for $1\leq n\leq k$
and $q(k+1) = q(k+1 - q(k)) + f(k+1)$, $q'(k+1) = q'(k+1 - q'(k)) + f'(k+1)$,
Also, $k+1 - q(k) = k+1 - q'(k)$, but $f(k+1)\neq f'(k+1)$.
Hence $q(k+1)\neq q'(k+1)$ and so the sequences $q$, $q'$ differ.
\end{proof}
We remark here that the recursion~\eqref{rec} can be used in either
direction. That is, given any sequence $f$ for which $q = Q(f)$ exists,
sequence $q$ can obviously be computed, term-by-term, in order of increasing index,
from~\eqref{rec}; and, given any sequence $q$ obeying $1\leq q(n)\leq n$
for all $n$,~\eqref{rec} can be used to compute $f(n)$ for all $n$. A case where computing
$f$ given $q$ gives an interesting result is
\begin{lemma}
\label{qhasD}
Let sequence $q(n)$ obey (i) $1\leq q(n)\leq n$ and (ii) $q(n+1) - q(n)\in\{0, 1\}$,
both for $n\in\mathbb N$, and let $q(1) = 1$. Define $f(n) = q(n) -
q(n-q(n-1))$ for $n\geq 2$ with $f(1) = 0$. Then, for $n\geq 1$, 
$f(n+1) - f(n)\in\{-1, 0, 1\}$.
\end{lemma}
\begin{proof}
Throughout the proof, we assume integer $n\geq 2$.
Assumption (i) tells us that $q$ exists and (ii), that $q$ has property D.
By the definition of $f(n)$,
$$f(n+1) - f(n) = \underbrace{q(n+1) - q(n)}_{:=A(n)} - 
\underbrace{\left[q(n+1 - q(n)) - q(n-q(n-1))\right]}_{:= B(n)}.$$
Assumption (ii) implies immediately that $A(n)\in\{0, 1\}$.

Now define $\ell(n) := n - q(n-1)$: by (i), $1\leq\ell(n)\leq n-1$. Then
$\ell(n+1) - \ell(n) = 1 - (q(n) - q(n-1))\in\{0, 1\}$, again by (ii).
Hence, $B(n)$ is either $q(\ell(n))-q(\ell(n)) = 0$, or $q(\ell(n)+1)-q(\ell(n))\in\{0,1\}$.
Therefore, $B(n)\in\{0,1\}$. Thus we have $A(n)-B(n)\in\{-1, 0, 1\}$ and
the lemma is proven.
\end{proof}
Note that the converse of lemma~\ref{qhasD} is not true: if
$f(n+1)-f(n)\in\{-1, 0, 1\}$ then $q$ does not necessarily have property D:
a counterexample is given by $f_b$ in lemma~\ref{nonFq}.

Finally, we point out that cases in which $f\in\F$ gives rise to a monotonic sequence
$q$ appear to be rare --- a few examples, each of which also has property \D,
are given in section~\ref{q_has_D}.

\section{The main theorem}

We now give a proof of Theorem~\ref{main}. We need several lemmas and we start with
\begin{lemma}[The shift property]
Let $\sinf{f}$, with $f(1) = 0$, be any sequence of integers for which the corresponding
sequence $\sinf{q} = Q(f)$, with $q(1) = 1$, exists. Define 
$$(\mathrm{A}):\;\; \sinf{f'}:= \begin{cases} 0 & n = 1\\ f(n-1) & n > 1\end{cases}
\;\;\;\mbox{and}\;\;\;
(\mathrm{B}):\;\;\sinf{q'}:= \begin{cases} 1 & n = 1\\ q(n-1) & n > 1\end{cases}.$$
Then $(\mathrm{A})$ implies $(\mathrm{B})$ and $(\mathrm{B})$ implies $(\mathrm{A})$.
\label{shift}
\end{lemma}
\begin{proof}
{\bf A $\Rightarrow$ B.} We use strong induction and the fact that
both $q$ and $q'$ obey~\eqref{rec}.

First, $q(1) = q'(1) = 1$ by definition, and $q'(2) = q'(2-q'(1))+f'(2) = 1 + f(1) = 1$.
Hence, $q'(2) = q(1)$.

Next, assume that 
\begin{equation}
\label{assump}
q'(n) = q(n-1) \qquad\mbox{for}\qquad n = 3, \ldots, k\mbox{ with } k > 3.
\end{equation}
Now, by assumption~\eqref{assump}, we have that
$q'(k) = q(k-1)$. For the inductive step $k\mapsto k+1$, we have that
\begin{eqnarray*}
q'(k+1) - q(k) &=&  q'(k+1 - q'(k)) - q(k - q(k-1))\\
& =& q'(k+1 - q(k-1)) - q(k-q(k-1)) =  q'(m+1) - q(m),
\end{eqnarray*}
where $m = k-q(k-1)$.
Since $q$ exists, $2\leq m+1 \leq k$
\footnote{It is in the light of this inequality that \textit{strong} induction is used.}
and so, by~\eqref{assump}, $q'(m+1) = q(m)$ and hence $q'(k+1) = q(k)$. 

{\bf B $\Rightarrow$ A.} This is straightforward. By definition, we have
$f'(1) = 0$. Then, by B, we have
$q'(1) = 1$ and $q'(2) = q(1) = 1$. Hence, by~\eqref{rec}, $f'(2) = q'(2) - q'(2-q'(1)) = 0$.
Then, letting $n\geq 3$, we have
$$f'(n) = q'(n) - q'(n-q'(n-1)) \underrel{\mbox{\scriptsize{by B}}}{=} q(n-1) - q(n-1-q(n-2))
\underrel{\mbox{\scriptsize{by~\eqref{rec} }}}{=} f(n-1)$$
as claimed.
\end{proof}

We will also need the following two special cases of $f\in\F$, each of which,
unusually, leads to a sequence $q$ with property D:
\begin{lemma}
\label{simp_q}
Let $f_0(n) = 0$ and $f_1(n) = n-1$, both for all $n\in\mathbb N$, and define
sequences $q_0 = Q(f_0)$, $q_1 = Q(f_1)$. Then $q_0(n) = 1$ and $q_1(n) = n$,
also for all $n$.
\end{lemma}
\begin{proof}
For $q_0(n)$,~\eqref{rec} gives $q_0(n) = q_0(n-q_0(n-1))$ since $f_0(n) = 0$.
By induction, with base case $q_0(1) = 1$, we make the
assumption that for some $k>1$ we have $q_0(k) = 1$. Then, by the
recursion formula, we have $q_0(k+1) = q_0(k+1-q_0(k)) = q_0(k) = 1$, 
as claimed.

For $q_1(n)$,~\eqref{rec} gives $q_1(n) = q_1(n-q_1(n-1))+n-1$ since $f_1(n) = n-1$.
With $q_1(1) = 1$, we assume that $q_1(k) = k$ for some $k > 1$. Then
$$q_1(k+1) = q_1(k+1 - q_1(k)) + k = q_1(1)+k = k+1$$
and the proof is complete.
\end{proof}
In the remainder of this section, $f$ will always have property \Dz\ and
we will assume that $f$ is such that $Q(f)$ exists. We introduce the handy notation
$\{j:k\}$, with $k\geq j$, to designate the set of successive integers $\{j, j+1,\ldots, k\}$.

We now consider the question: for a particular $n$, given that $f(n) = i\in\{0:n-1\}$,
what are the possible values of $q(n)$? That is, we investigate
the sets $\{q(n)|_{f(n) = i}\}$, where $f$ ranges over $\F_n$.
The lower bound on each of these sets is given by
\begin{lemma}
\label{min_T}
If $Q(f)$ for $f\in\F$ exists, then
(i) for $n\in\mathbb N$ and $i= 0, \ldots,n-1$, $\min\left(\{q(n)|_{f(n) = i}\}\right) = i+1$; and
(ii) a sequence $f$ that gives rise to this minimum value is
$$\left(f(k)\right)_{k=1,\ldots n} = (\underbrace{0, \ldots, 0,}_{n-i\;\mbox{\scriptsize zeroes}}\, 1, 2, \ldots, i).$$
\end{lemma}
\begin{proof}
For (i), if $n=1$ then, by the initial conditions, $f(1) = 0$
and $q(1) = 1 = f(1)+1$; and so the claimed lower bound holds. Therefore,
fix $n>1$ and recall that $q(n) = q(n-q(n-1))+f(n) = q(n-q(n-1))+i$. Then, provided that
$q$ exists, $1\leq q(k)\leq k$ for $1\leq k \leq n$. In particular $1\leq q(n-1)\leq n-1$
and so $1\leq n - q(n-1)\leq n-1$. Recall now that we are considering $f(n) = i$.
Hence, $q(n)= q(j)+i$, where $1\leq j\leq n-1$, so $1\leq q(j)\leq n-1$.
Therefore, $q(n) \geq i+1$, and this gives the required minimum.

For (ii), we use lemmas~\ref{shift} and~\ref{simp_q}. In order that $f(n) = i$,
we must have $i\leq n-1$, otherwise $f$ cannot have property \Dz. If
$f(n) = n-1$, then, by lemma~\ref{simp_q}, $q(n) = n$ can be achieved by choosing $f = (0, 1, \ldots, n-1)$.
Otherwise, $f(n) = i < n-1$ and in this case, the shift
property of lemma~\ref{shift} applied $(n-i-1)$ times gives $q(n) = i+1$.
\end{proof}
We have easily found the minimum of the sets $\{q(n)|_{f(n) = i}\}$, but the
corresponding maximum is not so straightforward. There is, however, a way round
the problem. We require the definition of two collections of 
finite sets of integers, $\T$ and $\U$, both of which have identical
structure. Taking $\T$ first, this is the collection of sets
$T_{i,n}$, where $n\in\mathbb N$ and $i = 0,\ldots, n-1$, which can 
be pictured as a triangular arrangement with $n$ indexing rows and
$i$ indexing the position within row $n$.  The definition of $\T$ is
\begin{definition}
\label{Tdef}
$\T$ is the collection of sets
$$T_{i, n} := \left\{q(n) : q = Q(f), f\in\F \mbox{ such that } f(n) = i\right\},$$
with $n\in\mathbb N$ and $0\leq i\leq n-1$.
\end{definition}
By lemma~\ref{simp_q} we have 
\begin{equation}
\label{T0n}
T_{0,n} = \{1\}\quad(\mathrm{a})\qquad\mbox{ and }\qquad T_{n-1,n} = \{n\}\quad(\mathrm{b}),
\end{equation}
for $n\in\mathbb N$.
For (a), if $f\in\F$ and $f(n) = 0$, then it must be that $f(i) = 0$ for $1\leq i < n$
and so the first part of lemma~\ref{simp_q} applies.
For (b), if $f\in\F$, the only way that $f(n) = n-1$ is if $f(i) = i-1$ for $1\leq i < n$,
so the second part of lemma~\ref{simp_q} applies.

We have not succeeded in computing $\T$ explicitly for all valid $i$ and
$n$. In principle, $\T$ can be computed by finding all the values that $q(n)$ can take given that
$f\in\F$ and $f(n) = i$. This would in turn require knowledge of all the values that $q(n-q(n-1))$ can take
(and then adding $i = f(n)$ to them). In practice, though, it seems that we have
insufficient knowledge of the sequences $q = Q(f)$ as $f$ ranges over $\F$.

However, we can explicitly construct a set, $\U$, with identical structure to $\T$,
and which will turn out to contain $\T$. This is sufficient to prove theorem~\ref{main}. 
The collection $\U$ consists of the sets
\begin{equation}
\label{Udef}
U_{i,n} := \begin{cases}
\{1\} & i = 0\\
\{i+1:n\} & i = 1, \ldots, n-1.
\end{cases}
\end{equation}
for $n\in\mathbb N$.

We now prove
\begin{lemma}
\label{TeqU}
Let collections of sets $\T$ and $\U$ be as in definition~\ref{Tdef} and
equation~\eqref{Udef} respectively. Then
$T_{i,n}\subset U_{i,n}$ 
for $n\in\mathbb N$ and $i = 0, \ldots, n-1$.
\end{lemma}
\begin{proof}
First, note that for $i=0$,~\eqref{T0n}(a) applies and gives $T_{0,k} = \{1\} =  U_{0,k}$. Furthermore, 
for $i=k$,~\eqref{T0n}(b) applies, giving $T_{k-1,k} = \{k\} = U_{k-1,k}$.
Both of these are true for all positive $k$.

The rest of the proof proceeds by strong induction, starting from $T_{0,1} = \{1\} = U_{0,1}$.
We assume that $T_{i,n} \subset U_{i,n}$ for $n = 1, \ldots, k$
and all allowed values of $i$, which assumption we refer to as (H). We then study what 
happens when $k\mapsto k+1$. The lemma having been proved for $i=0$ and $i=k$,
it only remains to study the case $1\leq i \leq k-1$.

Let us therefore consider $T_{i,k+1}$ with $1\leq i\leq k-1$. Equation~\eqref{rec} now reads 
$q(k+1) = q(k+1-q(k)) + i$ since, by definition, $f(k+1) = i$. We first need the
set of possible values of $\left.q(k)\right|_{f(k+1)=i}$, which, by (H), is 
$T_{i-1,k}\cup T_{i,k}\subset U_{i-1,k}\cup U_{i,k} = \{i:k\}$.
Only sets with indices $(i-1,k)$ and $(i,k)$ are allowed because $f$ has property \Dz:
$f(k+1)=i$ can only be true if $f(k) = i-1$ or $f(k)=i$.
Thus, $k+1-q(k)\in\{1:k-i+1\}$. Note that $2\leq k-i+1\leq k$, since $1\leq
i\leq k-1$, so $q(k+1-q(k))$ is defined only in terms of $q$ with argument
strictly less than $k+1$. Hence, informally put, $q(k+1-q(k))\in$ \{union of rows 1 to
$k-i+1$ of $\T\}$; that is, using assumption (H) and~\eqref{Udef},
\begin{equation}
\label{unions}
\left.q(k+1-q(k))\right|_{f(k+1)=i}\in\bigcup_{n=1}^{k-i+1}\bigcup_{j=0}^{n-1} T_{j,n}
\subset\bigcup_{n=1}^{k-i+1}\bigcup_{j=0}^{n-1} U_{j,n} = \{1:k-i+1\}.
\end{equation}
Therefore, for $1\leq i\leq k-1$, we have $q(k+1)\in\{i+1:k+1\}$ and so
$T_{i,k+1}\subset\{i+1:k+1\}$. Hence, from~\eqref{Udef}, we have $T_{i,k+1}\subset U_{i,k+1}$
and this completes the proof.
\end{proof}
\begin{corollary}
\label{bound}
Hence, for $n\in\mathbb N$, $q(n)\in\bigcup_{j=0}^{n-1} U_{j,n} = \{1:n\}$.
\end{corollary}
Lemma~\ref{TeqU}, in particular corollary~\ref{bound}, shows that $1\leq q(n)\leq n$ for all
$n\in\mathbb N$ and so Theorem~\ref{main} is proven.

We return to $\T$. In words, $T_{i, n}$ is the set of all values that $q(n)$ 
attains in practice, given that $f(n)$, as $f$ ranges over $\F$, is equal to $i$. Clearly, since $f$ has
property \Dz, $0\leq i=f(n) \leq n-1$. It is important to note that
$\T$ is not equal to $\U$ because approximations were made in the proof of
lemma~\ref{TeqU} in order to compute the values that $q(n-q(n-1))$ can assume \textit{in principle}.
In particular, the unions in~\eqref{unions} are typically over more
--- and larger --- sets $U_{i,n}$ than necessary.
These approximations almost always overestimate, and never underestimate $U_{i,n}$, giving a collection of sets
$\U$ such that those in $\T$ are contained within the corresponding sets in $\U$:
that is $T_{i,n}\subset U_{i,n}$.

As stated earlier, $\T$ can be visualised as a triangular array of sets --- see figure~\ref{Tfig}.
This was computed by brute force. For instance, the last row corresponds to
$n=8$ and was computed by generating all $2^7$ sequences $f\in\F_8$,
finding the sequences $q$ corresponding to each and noting the different
values of $\left. q(8)\right|_{f(8) = i}$, for $0\leq i \leq 7$.
As an example, the fourth row of figure~\ref{Tfig} implies that, for instance,
\begin{itemize}
\item If $f(4) = 0$, $q(4)$ can only be 1, that is, $T_{0,4} = \{1\}$. This is obvious since 
the sequence $f$, $f\in\F_4$, can only be $(0, 0, 0, 0)$. See lemma~\ref{simp_q}.
\item Now suppose that $f(4) = 1$. In this case, the allowed
sequences $f$ are $(0,0,0,1)$, $(0,0,1,1)$ and $(0,1,1,1)$. Using~\eqref{rec}
to compute $q(4)$ in each case, we find that the first two give $q(4) = 2$ and the third
one, $q(4) = 3$.
Thus, $T_{1,4} = \{2, 3\}$ (cf.\ $U_{1,4} = \{2,3,4\}$).
\end{itemize}

\begin{figure}[htbp]
\centering
\begin{picture}(150,63)
\put(4,60){$n$}
\put(77,60){$T_{i,n}$}
\put(0,58.4){\line(1,0){150}}
\put(77,54){$\{1\}$}
\put(70,47){$\{1\}\qquad\{2\}$}
\put(62.7,40){$\{1\}\qquad\{2\}\qquad\{3\}$}
\put(49.5,33){$\{1\}\qquad\{2:3\}\qquad\{3:4\}\qquad\{4\}$}
\put(40,26){$\{1\}\qquad\{2:3\}\qquad\{3:4\}\qquad\{4:5\}\qquad\{5\}$}
\put(30,19){$\{1\}\qquad\{2:3\}\qquad\{3:4\}\qquad\{4:5\}\qquad\{5:6\}\qquad\{6\}$}
\put(19.8,12){$\{1\}\qquad\{2:4\}\qquad\{3:5\}\qquad\{4:6\}\qquad\{5:7\}\qquad\{6:7\}\qquad\{7\}$}
\put(9.4,5){$\{1\}\qquad\{2:4\}\qquad\{3:6\}\qquad\{4:7\}\qquad\{5:7\}\qquad\{6:8\}\qquad\{7:8\}\qquad\{8\}$}
\put(4,54){$1$}
\put(4,47){$2$}
\put(4,40){$3$}
\put(4,33){$4$}
\put(4,26){$5$}
\put(4,19){$6$}
\put(4,12){$7$}
\put(4,5){$8$}
\end{picture}
\caption{The first eight rows of the triangular array of sets $\T$. Here, $i$ and $n$ index the
sets horizontally and vertically respectively, with $i = 0,\ldots, n-1$. The
set of consecutive integers $k,\ldots, l$ is written $\{k:l\}$. }
\label{Tfig}
\end{figure}

\section{Examples of sequences $q$ that have property D}
\label{q_has_D}
In lemma~\ref{simp_q}, we gave the examples $f(n) = 0$ and $f(n) = n-1$,
both of which lead to sequences $q$ that also have property D.
We are aware of three other examples, which we now present.

The first example is given in
\begin{lemma}
\label{n_over_4}
$$\mbox{If } f(n) = \Bigl\lfloor{\frac{n+2}{4}}\Bigr\rfloor \mbox{ then }
q(n) = \Bigl\lfloor{\frac{n+2}{2}}\Bigl\rfloor.$$
\end{lemma}
\begin{proof}
Let
$$\tilde{q}(x) = \frac{x}{2} + \frac{3 + \cos\pi x}{4}$$
for $x\in\mathbb R$. Then it can easily be verified that $\tilde{q}(n) =
\lfloor\frac{n+2}{2}\rfloor$ for $n\in\mathbb N$. Furthermore, direct
calculation gives
$$\tilde{q}(x) - \tilde{q}(x-\tilde{q}(x-1)) = \frac{1}{8}(2x+1+\cos\pi x) -
\frac{1}{4}\sin\frac{\pi}{4}(2x + 1 + \cos\pi x) := \tilde{f}(x).$$
Finally, it is straightforward to check that, for $n\in\mathbb N$, $\tilde{f}(n) = \lfloor\frac{n+2}{4}\rfloor$
and this proves the lemma.
\end{proof}
\begin{remark}
We have proved more here, viz.\ that $\tilde{q}(x)$, $\tilde{f}(x)$ as given above obey
$\tilde{q}(x) = \tilde{q}(x-\tilde{q}(x-1)) + \tilde{f}(x)$ for all $x\in\mathbb R$:
a continuous solution to~\eqref{rec}. The same is true of $f_0(x) = 0$,
which implies that $q_0(x) = 1$, and $f_1(x) = x-1$, which implies that $q_1(x) = x$, 
both for $x\in\mathbb R$ --- an extension of lemma~\ref{simp_q}.
\end{remark}

For the second example, define 
\begin{equation}
\label{delta_def}
\delta(n) := \begin{cases} 1, & n=0\\ 0, & \text{otherwise,}\end{cases}
\end{equation}
for $n\in\mathbb Z$. Another instance of a sequence $q$ with property D
is found when $f(n) = 1 - \delta(n-1)$. A quick calculation gives the
values in table~\ref{f1}.
\begin{table}[h]
\centering
\begin{tabular}{|l|c|cc|ccc|cccc|ccccc|lc|} \hline
$n$        & 1 & 2 & 3 & 4 & 5 & 6 & 7 & 8 & 9 & 10 & 11 & 12 & 13 & 14 & 15 & 16 &\ldots\\ \hline
$q(n)$     & 1 & 2 & 2 & 3 & 3 & 3 & 4 & 4 & 4 & 4 & 5 & 5 & 5 & 5 & 5 & 6 &\ldots\\ \hline
$n-q(n-1)$ & -- & 1 & 1 & 2 & 2 & 3 & 4 & 4 & 5 & 6 & 7 & 7 & 8 & 9 & 10 & 11 & \ldots\\ \hline
\end{tabular}
\medskip
\caption{The first 16 terms of $q$ when $f(n) = 1 - \delta(n-1) = (0, 1, 1,\ldots)$.}
\label{f1}
\end{table}

The values of $n$ between which $q(n)$ increases are marked by vertical lines.
Note in particular from this table that the sequence $q$ has property D, at least up to $n = 16$.
The assumption that the pattern seen in the table is maintained indefinitely leads us to conjecture that
$$\mbox{For $k\in\mathbb N$, } q(n) = k \mbox{ for } h(k) \leq n \leq h(k+1)-1,$$
where $h(k) := \frac{k^2-k+2}{2}$. 

Bearing these observations in mind, we now prove
\begin{lemma}
\label{root2n}
Let $f(n) = 1 - \delta(n-1) = (0, 1, 1,\ldots)$. Then
$q(n) = (1, 2, 2, 3, 3, 3, 4,\ldots)$, where each positive integer appears in
turn, with integer $k$ occurring $k$ times in succession. That is,
$$q(n) = k\qquad\mbox{for}\qquad h(k)\leq n \leq h(k+1)-1,$$
where $k\in\mathbb N$ and $h(k) =  \frac{k^2-k+2}{2}$.
\end{lemma}
\begin{proof}
This sequence~\cite{oeis_d}, is a special case of a generalised Golumb triangular recursion~\cite{ikt}.
Once again, the proof proceeds by induction.
Calculation gives $q(1) = 1$, $q(2) = q(3) = 2$, so set $k>2$ and make the hypothesis
\begin{equation}
\label{H1}
q(n) = k\qquad\mbox{for}\qquad h(k)\leq n\leq h(k+1)-1.\tag{H1}
\end{equation}
We show that this implies that
$$q(n) = \begin{cases} k+1 & \mbox{for} \qquad h(k+1)\leq n\leq h(k+2)-1, \mbox{ and}\\
k+2 & \mbox{for}\qquad n = h(k+2).\end{cases}$$
Let $m = h(k+1)$. Then~\eqref{rec} gives
$q(m) = q(m-q(m-1))+1 = q(m-k)+1$ by~\eqref{H1}.
Now, $m-k = \frac{k^2-k+2}{2} = h(k)$, so, by~\eqref{H1}, $q(m-k) = k$ and
hence $q(m) = k+1$.

Now consider $q(m+i+1)$ for $0\leq i \leq k-1$. Starting from base case $q(m)=k+1$, we
use induction again, this time on $i$, to prove that 
\begin{equation}
\label{H2}
q(m+i) = k+1 \Rightarrow q(m+i+1) = k+1 \qquad\mbox{for}\qquad 0\leq i \leq k-1\tag{H2}.
\end{equation}
By~\eqref{rec}, we have $q(m+i+1) = q(m+i+1-q(m+i))+1 = q(m-k+i)+1$,
the last by~\eqref{H2}. Now, $m-k+i = \frac{k^2-k+2}{2}+i = h(k)+i$, and since
$0\leq i \leq k-1$, we have $q(m-k+i) = k$, by~\eqref{H1}.
Hence,~\eqref{H2} is proven.

Finally, consider $q(m+k+1) = q(m+k+1-q(m+k))+1 = q(m)+1 = k+2$, which
completes the proof.
\end{proof}
\begin{remark}
This result can also be written as
$$f(n) = 1-\delta(n-1)\Rightarrow q(n) = \Bigl\lfloor{\frac{1}{2} + \sqrt{2n - \frac{7}{4}}}\Bigr\rfloor =: w(n),$$
as can be easily shown by considering $\sqrt{2 h(k)-7/4} = k-1/2$ for $k\in\mathbb N$.
However, we lack a direct proof of lemma~\ref{root2n} starting from this expression.
\end{remark}
\begin{remark}
This result relates immediately to the fact that $T_{1,n}\subsetneq U_{1,n}$:
we have explicitly constructed $T_{1,n} = \{2: w(n)\}$ above, from which it is clear that $T_{1,n}\subsetneq
U_{1,n}$ for $n >2$. See the second top-right to bottom-left diagonal in figure~\ref{Tfig}.
\end{remark}
The third example of a sequence $q$ with property D is especially interesting and involves
$\gamma$, the reciprocal of the golden ratio:
\begin{theorem}
\label{mysterious}
Let $\gamma := \frac{\sqrt{5}-1}{2}$, so that $\gamma^2 + \gamma - 1 = 0$.
Then $q(n) = 1 +\lfloor\gamma(n-1)\rfloor$ obeys~\eqref{rec} with
$f(n) = \lfloor\gamma^2 n\rfloor$.
\end{theorem}
Since $\gamma \approx 0.618 < 1$, clearly $q(n)$ as defined here has
property D, and, furthermore, $q(1) = 1$. To prove the theorem, we first show that 
it is equivalent to an identity,~\eqref{the_id}, which we then prove in
lemma~\ref{proof_id}.

In the course of the proof of this lemma, we need a few small results:
\begin{enumerate}[label={(\roman*)}]
\item $\gamma^2 + \gamma = 1$, which implies that $2\gamma^2 + 3\gamma = 2 +
\gamma$, $\gamma^{-1} = 1 + \gamma$, $\gamma^2(\gamma+2) = 1$.
\item $x = \lfloor x\rfloor + \fr{x}$, $x\in\mathbb R$, where $\fr{x}$ is the fractional part of $x$
\item $\lfloor -x\rfloor = -\lfloor x\rfloor - 1$, $x\notin\mathbb Z$
\item $\lfloor m+x\rfloor = m+\lfloor x\rfloor$ for $m\in\mathbb N$
\item $\fr{\gamma^2 n} = 1 - \fr{\gamma n}$.
\end{enumerate}
Items (i) come from the given definition of $\gamma$ and (ii) -- (iv)
are from the definition of the floor function. For (v), start from (ii)
with $x=\gamma^2 n$, which gives $\fr{\gamma^2 n} = \gamma^2 n - \flgg{n}$.
Now use (i) to replace $\gamma^2$ with $1-\gamma$, giving 
$\fr{\gamma^2 n} = n - \gamma n - \lfloor{n - \gamma n}\rfloor = -\gamma n -
\lfloor{-\gamma n}\rfloor$, where (iv) has been used. Then (iii) gives
$\fr{\gamma^2 n} = -\gamma n + \lfloor{\gamma n}\rfloor+1 = 1-\fr{\gamma n}$,
using (ii).

Starting from $q(n) = 1 + \flg{(n-1)}$, clearly $q(1) = 1$. It is convenient to shift $n$ by 1,
giving $q(n+1) = 1 + \flg{n}$, which we will show obeys~\eqref{rec} for
$n\in\mathbb N$.  When this is substituted in $q(n+1) - q(n+1-q(n))$ we have
$$q(n+1) - q(n+1-q(n)) = \flg{n} - \lfloor{\gamma(n-1) - \gamma\flg{(n-1)}}\rfloor.$$
Now, using (i), (iii) and (iv), we find that
$$q(n+1) - q(n+1-q(n)) = n-1 - \flgg{n} - \lfloor{\gamma +\gamma\flgg{(n-1)}}\rfloor.$$
Hence, theorem~\ref{mysterious} is true if the right-hand side of the above
expression is equal to $f(n+1) = \flgg{(n+1)}$ for $n\in\mathbb N$, that is, if
\begin{equation}
\label{the_id}
\theta(n) := \lfloor{\gamma +\gamma\flgg{(n-1)}}\rfloor + \flgg{n} + \flgg{(n+1)} = n-1
\;\;\mbox{ for }\;\; n\in\mathbb N.
\end{equation}

This identity~\eqref{the_id} is proved in
\begin{lemma}
\label{proof_id}
If $\gamma$ is as defined in theorem~\ref{mysterious}, then $\theta(n)$,
as defined in equation~\eqref{the_id}, obeys $\theta(n) = n-1$ for all $n\in\mathbb N$.
\end{lemma}
\begin{proof}
We consider separately three cases which are demarcated according to
$$\mbox{Case A: }\fr{\gamma^2 n}\in(0, \gamma^2);\qquad
\mbox{Case B: }\fr{\gamma^2 n}\in(\gamma^2, \gamma);\qquad
\mbox{Case C: }\fr{\gamma^2 n}\in(\gamma, 1).$$
The intervals are all open because there is no $n\in\mathbb N$ such that 
$\fr{\gamma^2 n} = 0$ or $\gamma$, and no $n>1$ such that
$\fr{\gamma^2 n} = \gamma^2$, all three because $\gamma^2 = (3-\sqrt{5})/2$ is irrational.
For instance, $\fr{\gamma^2 n} = \gamma$ implies that $\gamma^2 n = m
+\gamma = m + 1 - \gamma^2$, where $m$ is an integer, and hence, that
$\gamma^2 = (m+1)/(n+1)$ --- a contradiction. A parallel argument can be
applied to show the impossibility of $\fr{\gamma^2 n} = 0$ or $\gamma^2$, the latter with
$n>1$.  Finally, by definition, $\fr{x}<1$.

Proving cases A--C amounts primarily to rearranging $\theta(n)$ into a first-degree
polynomial in $n$ plus the floor of a bounded function of $n$.

{\bf Case A.} Note that $\fr{\gamma^2 n}\in(0,\gamma^2)$ implies that
$\flgg{(n-1)} = \flgg{(n)}-1$ and $\flgg{(n+1)} = \flgg{(n)}$. Thus, in
case A we have
$$\theta(n) = \lfloor{\gamma\flgg{n}}\rfloor + 2\flgg{n} =: \theta_A(n).$$
Consider the first term in $\theta_A(n)$:
$$\lfloor{\gamma\flgg{n}}\rfloor = \lfloor{\gamma\lfloor{(1-\gamma)n}\rfloor}\rfloor
= \lfloor{\gamma n - \gamma(\flg{n}+1)}\rfloor
= \lfloor{\gamma n - \gamma(\gamma n -\fr{\gamma n}+1)}\rfloor, $$
where we have used (i), (iv), (iii) and then (ii). Replacing the first $\gamma n$ with
$(1-\gamma^2)n$, we get
$$\lfloor{\gamma\flgg{n}}\rfloor 
= n + \lfloor{-2\gamma^2 n +\gamma\fr{\gamma n}-\gamma}\rfloor
= n-1 - \lfloor{2\gamma^2 n -\gamma\fr{\gamma n} + \gamma}\rfloor,$$
using (iii). We can now use (ii) and (v) to obtain $\lfloor{\gamma\flgg{n}}\rfloor =$
$$ n-1 -\lfloor{2(\flgg{n} + \fr{\gamma^2 n}) +\gamma(1-\fr{\gamma n})}\rfloor
= n-1 -\lfloor{2(\flgg{n} + \fr{\gamma^2 n}) +\gamma\fr{\gamma^2 n}}\rfloor. $$
Hence, using (iv),
$$\theta_A(n) 
= n-1 -\lfloor{2(\flgg{n} + \fr{\gamma^2 n}) +\gamma\fr{\gamma^2 n}}\rfloor +2\flgg{n}
= n-1 - \lfloor{(\gamma+2)\fr{\gamma^2 n}}\rfloor, $$
and $\theta_A(n)$ is now in the required form.

Since $0 < \fr{\gamma^2 n} < \gamma^2$, we have, in case A,
$0 < (\gamma+2)\fr{\gamma^2 n} < \gamma^2(\gamma+2) = 1$,
using (i), and so $\lfloor{(\gamma+2)\fr{\gamma^2 n}}\rfloor = 0$.
Therefore, $\theta_A(n) = n-1$.

{\bf Case B.} For case B, we have $\gamma^2 <\fr{\gamma^2 n}< \gamma = 1-\gamma^2$
and this implies that $\flgg{(n-1)} = \flgg{n} = \flgg{(n+1)}$. Hence, in case B
$$\theta(n) = \lfloor{\gamma+\gamma\flgg{n}}\rfloor + 2\flgg{n} =: \theta_B(n).$$
Since $2\flgg{n}$ is an integer, we immediately have $\theta_B(n) = \lfloor{\gamma+(\gamma+2)\flgg{n}}\rfloor$.
Now,
$$(\gamma+2)\flgg{n} = (\gamma+2)(\gamma^2 n - \fr{\gamma^2 n})
= n -(\gamma+2)\fr{\gamma^2 n}, $$
by (ii) and (i). Hence, 
\begin{equation}
\label{thB}
\theta_B(n) = n + \lfloor{\gamma-(\gamma+2)\fr{\gamma^2 n}}\rfloor
= n-1 - \lfloor{(\gamma+2)\fr{\gamma^2 n}-\gamma}\rfloor,
\end{equation}
using (iii). Now, by the definition of case B, 
$$\gamma^2(\gamma+2) < (\gamma+2)\fr{\gamma^2 n} < \gamma(\gamma+2)
\qquad\mbox{or}\qquad
1 < (\gamma+2)\fr{\gamma^2 n} < \gamma+1.$$
Hence, $\lfloor{(\gamma+2)\fr{\gamma^2 n}-\gamma}\rfloor = 0$ and
$\theta_B(n) = n-1$.

{\bf Case C.} Case C is similar to case B. We now have $\gamma <\fr{\gamma^2 n}< 1$, giving
$\flgg{(n-1)} = \flgg{n}$ but $\flgg{(n+1)} = \flgg{n}+1$. Hence, in case C
$$\theta(n) = \lfloor{\gamma+\gamma\flgg{n}}\rfloor + 2\flgg{n}+1 =: \theta_C(n).$$
From the definition of $\theta_B(n)$, we see immediately that $\theta_C(n) = \theta_B(n)+1$ and
so, from~\eqref{thB}, we have
$\theta_C(n) =  n - \lfloor{(\gamma+2)\fr{\gamma^2 n}-\gamma}\rfloor$.
Hence, for case C,
$$\gamma(\gamma+2) < (\gamma+2)\fr{\gamma^2 n} < (\gamma+2)
\qquad\mbox{or}\qquad
1 < (\gamma+2)\fr{\gamma^2 n} - \gamma < 2.$$
Therefore, $\lfloor{(\gamma+2)\fr{\gamma^2 n}-\gamma}\rfloor = 1$ and so
$\theta_C(n) = n-1$, concluding the proof.
\end{proof}
The identity in lemma~\ref{proof_id} was shown to be equivalent to theorem~\ref{mysterious},
which we have therefore proved.

\section{Heuristic and experimental results}
\label{heuristic}

Nothing is proved in this section. Instead, heuristic plausibility arguments are used to give 
simple descriptions of the behaviour of $q$ for certain sequences $f$. Computation suggests
that the actual behaviour of $q$ follows our predictions closely,
as we shall see. We consider three types of sequence $f$.

\subsection{$\bm{f(n) = \lfloor\alpha n\rfloor}$, $\bm{\alpha\in(0, 1)}$}
The argument here is very simple: for large $n$, we approximate
$\lfloor\alpha n\rfloor$ by $\alpha n$ and use the ansatz $q(n) = an$,
where $a$ is to be found.
Substituting this in~\eqref{rec}, we obtain $a^2 n - a^2 = \alpha n$, from which,
neglecting the $-a^2$ term, we obtain $a = \sqrt\alpha$. This suggests the
approximation 
\begin{equation}
\label{linf}
q(n)\approx \sqrt\alpha n,
\end{equation}
which turns out to be very good. See figure~\ref{qhalf} for an example. This shows $q(n) -
n/\sqrt{2}$ when $f(n) = \lfloor n/2\rfloor$: for $n\leq 160000$, we find
from the data used to produce this figure that $|q(n)-n/\sqrt{2}| < 75.5$.

Theorem~\ref{mysterious} is a particular case of equation~\eqref{linf} with
$\alpha = \gamma^2$.

\subsection{$\bm{f(n) = a\mbox{ for } n > n'}$}
\label{f2const}
The case $f(n)$ tends to a constant $a\in\mathbb N$ as $n\to\infty$ is
interesting. We argue in reverse, proposing the ansatz $q(x) = \sqrt{2ax} - b$,
$x\in\mathbb R$, and then deducing the $f(x)$ that this implies. Starting from the
asymptotic expansion:
$$f(x) = q(x) - q(x-q(x-1)) \sim 
a + \frac{\sqrt{2 a} (a-2b)}{4\sqrt x} + \frac{a(a - 2b - 2)}{4x} + O\left(x^{-\frac{3}{2}}\right),$$
we let $b = a/2$ in order to eliminate the second term. This gives
$$f(x) \sim a - \frac{a}{2x} + O\left(x^{-\frac{3}{2}}\right).$$
This short calculation suggests that, for $a, n\in\mathbb N$,
\begin{equation}
\label{fconst}
f(n) = a - \lfloor{\delta_1(n)}\rfloor \implies q(n) = \sqrt{2an}-a/2 + \delta_2(n),
\end{equation}
where $\delta_1(1) = a$ and $\delta_i(n)\to 0$ as $n\to\infty$, $i=1,2$.

\begin{centering}
\begin{figure}[htbp]
\includegraphics[width=4.0in]{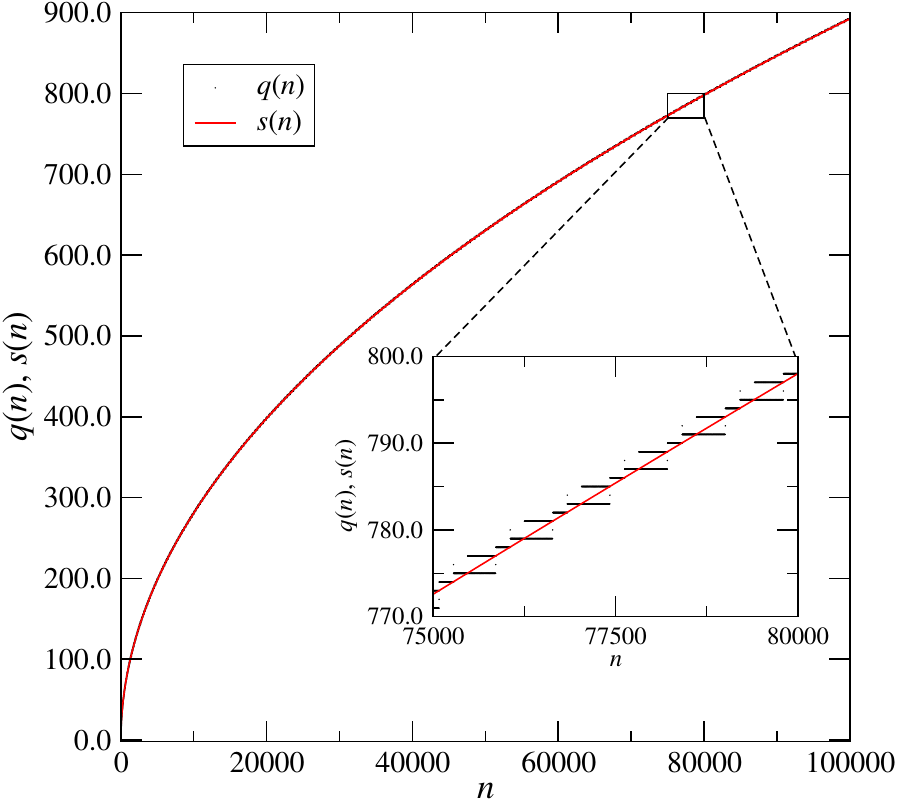}
\caption{Plot of $q(n)$, black dots, for $f(n) = \lfloor 5 - 5/\sqrt n\rfloor$, so
that $f(n)\to 4$ as $n\to\infty$. According to the argument in
section~\ref{f2const}, we expect that $q(n)\sim \sqrt{8n} - 2 := s(n)$,
this curve being plotted in red. The approximation appears to hold remarkably well,
the two plots being almost exactly superimposed.
Inset: enlargement of the region $75000\leq n \leq 80000$.}
\label{ascon}
\end{figure} 
\end{centering}

This appears to work surprisingly well --- see figure~\ref{ascon}, which
compares $q$ when $f(n) = \lfloor 5 - 5/\sqrt n\rfloor)$, so that
$\lim_{n\to\infty} f(n) = 4$, with $s(n) := \sqrt{8n} - 2$
for $n = 1, \ldots, 10^5$. Over this range, we find that $|q(n)-s(n)| < 2$.

Other examples were tried: 
$f(n) = \lfloor a - a\exp{(-b n)}\rfloor$,
$f(n) = \lfloor a - a/n^b\rfloor$ and 
$f(n) = \lfloor an\rfloor$ if  $n<n_0$ and $\lfloor a n_0\rfloor$ otherwise.
In all cases, $a, b$ were chosen so as to ensure that $f(n)$ had property
\Dz, and every time, \eqref{fconst} gave an excellent approximation to
$q(n)$: in these cases at least, it is the asymptotic behaviour
of $f(n)$ that determines the asymptotic behaviour of $q(n)$.

\subsection{$\bm{f(n) = \lfloor{c_1 n^{p_1}+c_2 n^{p_2} +\ldots}\rfloor}$ with $\bm{p_i\in(0,1)}$}
Arguing in reverse as in the previous case, we can obtain useful approximations to
$q(n)$ when $f(n)$ is the sum of fractional powers of $n$, at least in certain special cases.
We use the ansatz $q(x) = a x^p+b$ and calculate the asymptotic expansion for 
$$ \begin{aligned}
f(x) & = q(x) - q(x - q(x-1)) \sim ap\left[b\,x^{p-1} + \frac{b^2 (1 - p)}{2}\,x^{p-2} + \ldots\qquad\right. \\
& + a\,x^{2p - 1}
+ a (b(1-p) - p)\,x^{2p - 2}+\ldots\\
& \left. + \frac{a^2 (1 - p)}{2}\,x^{3p - 2}
+ \frac{a^2 (1-p)(b(2-p)-2p)}{2} \,x^{3p - 3}+\ldots\right].
\end{aligned} $$
Note that the powers of $x$ here are $p-i$, $i\in\mathbb N$ and $ip-j$,
$i\geq 2$, $j \geq i-1$. We cannot order the powers of $x$ unless a value of $p$ is
specified. Letting, for example, $a = 1$, $b=1/2$ and $p = 3/4$, we have
$f(x) = 3x^{1/2}/4 + 3 x^{1/4}/32 + 5/128 + O(x^{-1/4})$ and $q(x) = x^{3/4} + 1/2$.
Then, with $f(n) = \lfloor 3n^{1/2}/4 + 3 n^{1/4}/32 + 5/128\rfloor$, which has
property \Dz, we generate $q(n)$ via~\eqref{rec} and compare this with
our approximation $q'(n) := n^{3/4} + 1/2$. The agreement is good, with 
$-21\leq q(n) - q'(n)\leq 3$ for $1\leq n \leq N = 160000$, where $q(N) = 7990$.

\subsection{Dynamics} We close this section with some figures illustrating 
examples of the dynamics of $q$ for particular sequences $f$. 

\begin{centering}
\begin{figure}[htbp]
\includegraphics[width=4.0in]{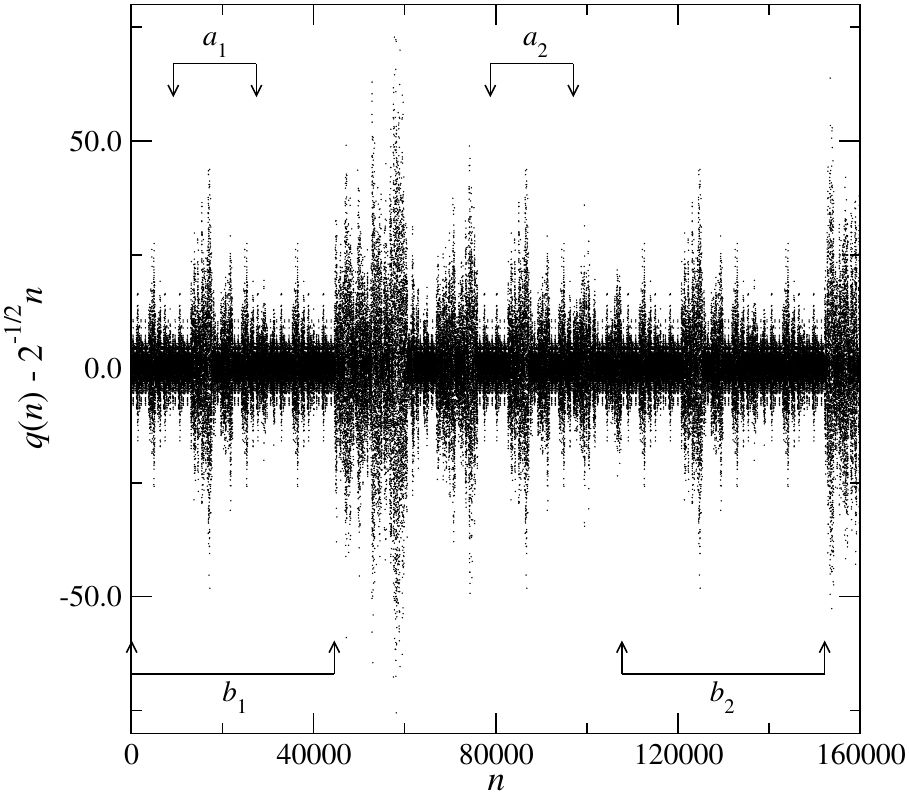}
\caption{Plot of $q(n) - n/\sqrt{2}$ for $n=1,\ldots, 160000$, with
$f(n) = \Big\lfloor\frac{n}{2}\Big\rfloor$. Two pairs of self-similar regions $a_1, a_2$
and $b_1, b_2$ are shown --- see text.}
\label{qhalf}
\end{figure} 
\end{centering}
An interesting case is $q = Q(\lfloor n/2\rfloor)$, shown in figure~\ref{qhalf}.
This shows the de-trended sequence $q(n)-n/\sqrt{2}$ --- see~\eqref{linf} --- for $1\leq n\leq 160000$.
At first sight, no obvious patterns appear in the plot, but in fact
there are several regions of exact self-similarity, even in this limited range of $n$. Specifically,
among others, we observe
\begin{itemize}
\item $q(i+69568)-q(i) = 49192\mbox{ for } i\in\{9235: 27465\}$ (regions $a_1, a_2$); and
\item $q(i+107616)-q(i) = 76096 \mbox{ for } i\in\{91: 44577\}$ (regions $b_1, b_2$).
\end{itemize}

Our final figure gives an intuitive idea of what happens if we make a small
perturbation. It is tempting to describe the behaviour of many of the
sequences $Q(f)$ as `chaotic'. However, since~\eqref{rec} is discrete in time and space, there is no
sense in which the system can be infinitesimally perturbed; perhaps the best we can do is
to compute sequences $q$, $q_1$, where $q_1$ is computed identically to $q$
--- using~\eqref{rec} in the usual way ---
except that we artificially adjust a single term, $q(n_1)$ say, for some small integer $n_1$. The
result of such a computation, with $q = Q(\lfloor n/2\rfloor)$, 
$q_1 = Q(\lfloor n/2\rfloor + \delta(n-16))$, so that $q_1(16) = 1+q(16)$, is shown in
figure~\ref{sens}. This figure is a plot of the difference $q(n) - q_1(n)$ versus $\log_2(n)$,
the scaling bringing out the approximate periodicity visible on a logarithmic scale.
We see regions where $q(n) - q_1(n) = 0$ exactly, interspersed with
regions where this difference is far from zero: the memory of the
perturbation seems to persist indefinitely.
\begin{centering}
\begin{figure}[htbp]
\includegraphics[width=4.0in]{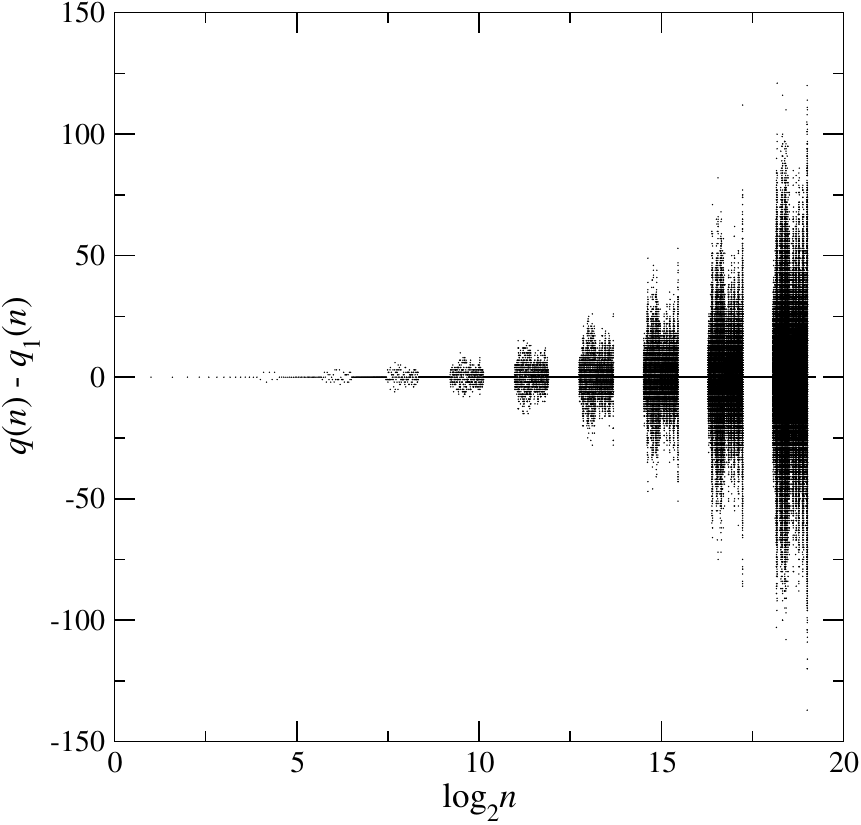}
\caption{Plot of the difference of $q(n) = Q(\lfloor n/2\rfloor)$ and 
$q_1(n) = Q(\lfloor n/2\rfloor + \delta(n-16))$. The idea is that $q_1(n)$
is a slightly perturbed version of $q(n)$. The effect of the perturbation persists, at
least until $n = 2^{19}$.}
\label{sens}
\end{figure} 
\end{centering}

\section{Conclusions and further work}

We have investigated the problem of the conditions on an integer sequence
$f(n)$, $n\in\mathbb N$, with $f(1) = 0$, such that the sequence $q(n)$,
with $q(1) = 1$, computed from $q(n) = q(n-q(n-1)) + f(n)$, exists. We
think of the sequence $q$ as a solution to this difference equation.
We have proved that if $f(n+1) - f(n)\in\{0, 1\}$, $n\geq 1$, then the solution $q$ exists, the
existence of $q$ being exactly equivalent to the inequality $1\leq q(n)\leq n$ holding for
all $n\in\mathbb N$. 

We define $\F$ as the set of semi-infinite sequences with differences 
between successive terms equal to 0 or 1.
In lemma~\ref{nonFq} examples were given of sequences $f\notin\F$ but for 
which $q$ nonetheless exists --- the condition that $f(n+1) - f(n)\in\{0,
1\}$ is sufficient but not necessary --- and this suggests that we define a second
set of semi-infinite integer sequences, $\E$, which is the set of all
sequences $f$ with $f(1) = 0$ such that the corresponding sequence $q$
exists. Clearly, $\F \subsetneq \E$.

It would be fruitful to investigate the structure of $\E$ further. The
question naturally arises as to whether the proof of theorem~\ref{main} 
could be extended to include
more sequences $f$. A better characterisation of $\E$ might be a
useful step on the way to settling the question of the existence of the
Hofstadter sequence, this problem being the original motivation for our work.
In fact, the question would have been settled by theorem~\ref{main}, were
the sequence $c(n) := q_h(n - q_h(n-2))$, with $q_h$ defined by~\ref{hof}, to obey
$c(n+1) - c(n)\in\{0, 1\}$ for $n\geq 3$. It does not --- in fact $c = (1,2,2,2,3,3,3,3,3,4,5,4,5,\ldots)$.

Solutions arising from $f\in\F$ typically appear to display non-trivial
dynamics: in particular, they are generally aperiodic and display no obvious patterns. For
some special sequences $f$, for instance $f(n) = \lfloor\alpha n\rfloor$
with $\alpha\in [0, 1)$, the \textit{average} behaviour appears to be well-defined,
however, and we give a heuristic argument for this. 

Less typical seem to be sequences $f\in\F$ for which the corresponding
solution is also monotonic, and we construct exact solutions in the cases of
which we are aware.

Both from the point of view of the existence of solutions and from the
study of their dynamics, the difference equation~\eqref{rec} appears to be
an interesting subject for further study, even in its own right.

\section*{Acknowledgement}
The Authors would like to acknowledge helfpul discussions with Dr Peter
Gallagher.

\end{document}